\documentclass[11pt,a4paper]{amsart}
\usepackage{commands_Ryo, packages_Ryo,packages,commands}

\title{On perfectoidization of finite algebras over a perfectoid ring}

\author[R. Ishizuka]{Ryo Ishizuka}
\address{Department of Mathematics, Institute of Science Tokyo, 2-12-1 Ookayama, Meguro, Tokyo 152-8551}
\email{ishizuka.r.ac@m.titech.ac.jp}

\author[L. Navarro Chafloque]{Léo Navarro Chafloque}
\address{EPFL SB MATH CAG,
MA C3 585 (Bâtiment MA),
Station 8,
1015 Lausanne}
\email{leo.navarrochafloque@epfl.ch}



\begin{document}

\begin{abstract}
    We study general properties of the perfectoidization of finite algebras over a perfectoid ring, which helps to understand some precise and explicit descriptions.
    For example, we prove that if $A=R[t]/(m(t))$ where $m(t)$ is monic, $R$ is perfectoid and the discriminant $d$ of $m(t)$ is a non-zero divisor of $R$ satisfying a bounded torsion condition, then $dA_{\pfd}\subset A$. 
    We also prove a density criterion reducing the construction of the perfectoidization to adjoining suitable \(p\)-power roots modulo \(p\). In the second part of the paper, we compute perfectoidizations in several families of examples, including Kummer-type extensions and split finite algebras.
\end{abstract}

\maketitle

\tableofcontents

\section{Introduction}


The goal of this paper is to understand more precisely the perfectoidization of finite algebras over a perfectoid ring. More precisely, if \(R\) is a perfectoid ring and \(R\to A\) is a finite finitely presented \(R\)-algebra, then \cite[Theorem 10.9]{bhatt2022Prismsa} implies that there exists an initial perfectoid \(A\)-algebra, which we denote by \(A_{\pfd}\). The paper has two complementary aims. First, we prove general results helping to understand perfectoidizations of finite algebras. Second, we compute perfectoidizations in explicit examples.


In \Cref{first_sec} we prove some theoretical results on the topic. Namely say that $R$ is a perfectoid ring and let's consider $R\to A=R[t]/(m(t))$ where $m(t)\in R[t]$ is a monic polynomial whose discriminant (see \Cref{disc_def}) $d\in R$ is a non-zero divisor. Then using \cite[Theorem 10.9]{bhatt2022Prismsa}, $A_{\pfd}$ is discrete and $A\to A_{\pfd}$ is an isomorphism away from $d$. Typically, $A_{\pfd}$ can fail to be finite over $R$. It is not obvious which elements have to be adjoined to \(A\) in order to obtain a perfectoid ring. We prove the following theorem, which sheds some light on the above question.

\begin{thm}[\Cref{boundeddtorsmonic}, \Cref{nzd_reassuring}]\label{thm_intro_discriminant}Let $R$ be a perfectoid ring and let $R\to A=R[t]/(m(t))$ where $m(t)\in R[t]$ is a monic polynomial whose discriminant $d\in R$ is a non-zero divisor and $R/dR$ is of bounded $p^{\infty}$-torsion. Then $A_{\pfd}$ is a subring of $A[1/d]$, in particular \(d\)-torsion-free, and moreover
\[dA_{\pfd}\subseteq A.\]
Equivalently any element of $A_{\pfd}$ can be written as $\frac{a}{d}$ for some $a\in A$.
\end{thm}

We also see the above as some kind of ``hidden finiteness'' of the perfectoidization: namely if $A_{\pfd}$ were a finite $R$-algebra, then the fact that $A\to A_{\pfd}$ is an isomorphism after inverting $d$ would imply $d^NA_{\pfd}\subseteq A$ for some $N\in \N$. 

 

Another result of \Cref{first_sec} gives a criterion for recognizing the perfectoidization from an intermediate algebra. Roughly speaking, it shows that it is enough to adjoin suitable compatible \(p\)-power roots modulo \(p\).

\begin{thm}[\Cref{densityprinciple}, \Cref{densityprincipleVp}]\label{densityprinciple_intro}
Let $R$ be a perfectoid ring
and $d\in R$ is a non-zero divisor.
Say that $R\to A$ is a finite and finitely presented map étale outside of $V(d)$ such that $d$ is also a non-zero divisor in $A$.
In this case, $A$ can be seen as a subring of $A_{\pfd}$.
Let $A \subseteq C \subseteq A_{\pfd}$ be an intermediate subring such that \(C/pC\) is semiperfect.

Then \(C\) is dense in \(A_{\pfd}\) for the subspace topology where $A_{\pfd}$ is equipped with the $p$-adic topology. If $p=d$, then the $p$-completion $C^{\wedge_p}$ is $A_{\pfd}$.

\end{thm}

Afterwards, in \Cref{sec_ex}, we produce examples of perfectoidizations of finite algebras over a perfectoid ring. Namely, a prototypical example in our study is the Kummer-type finite extension
\begin{equation}\label{equ_m_intro}
    R[t]/(t^m-d)
\end{equation}
for an $m\in \N$ with $(m,p)=1$. We prove the following --  this was also ``almost'' computed in the case $m=2$ in \cite[Example 2.3.12]{reinecke2020Moduli}.

\begin{thm}[\Cref{ComputePerfectoidizationTm}]
\label{ComputePerfectoidizationTmIntro} 
    Let \(R\) be a perfectoid ring. Let $m\in \N$ be an integer such that $(p,m)=1$. Let \(R\abracket{t^{1/p^\infty}}\) denote the \(p\)-adic completion of the perfect polynomial ring \(\bigcup_{n \geq 0} R[t^{1/p^n}]\) over \(R\).
    Then
    \begin{equation*}
       \left(R\abracket{t^{1/p^\infty}}[t^{1/m}]\right)_{\pfd}=R\abracket{t^{1/mp^\infty}} \defeq \parenlr{\bigcup_{n \geq 0} R[t^{1/mp^n}]}^{\wedge p}.
    \end{equation*}

\end{thm}

Using the universality of the above example, we get a general result in the case where $d\in R$ in \Cref{equ_m_intro} admits a $p$-compatible sequence of $p$-power roots in \Cref{ComputePerfectoidizationRootm}.

We also treat a complementary situation: the element \(d\) itself need not admit a compatible system of \(p\)-power roots, but a unit twist of \(d\) does. Under the additional assumption \(D(d)=D(p)\), we compute the perfectoidization using \Cref{densityprinciple_intro}.

\begin{thm}[\Cref{Computation_pi_unit}]
    \label{Computation_pi_unit_intro}
    Let \(R\) be a \(\pi\)-torsion-free perfectoid ring with $D(\pi)=D(p)$ in $\Spec(R)$ admitting a compatible system \(\{\pi^{1/p^n}\}_{n \geq 0}\) of \(p\)-power roots of \(\pi\).
    Take a unit \(u\) in \(R\) and an integer \(m\) such that \((m, p) = 1\).
    Set
    \begin{equation*}
        A \defeq R[t]/(t^m-\pi u),
    \end{equation*}
    which is finite \'etale outside \(\pi\) over \(R\).
    Write the multiplicative order \(k_0\) of \(p\) in \((\setZ/m\setZ)^\times\) and let $q=p^{k_0}$.
    Then the perfectoidization of \(A\) can be identified with
    \begin{equation*}
        A_{\pfd} = \parenlr{\bigcup_{n \geq 0} R\bracketlr{\parenlr{\pi^{1/q^n}u}^{1/m}}}^{\wedge_p}
    \end{equation*}
    where \((\pi^{1/q^n}u)^{1/m}\) is a \(m\)-th root of \(\pi^{1/q^n}u\) which we can take as an element of \(A[1/p]\).
\end{thm} 

As a third class of examples, which is also in general disjoint from the two others, we study the case $R[t]/m(t)$ when $m(t)$ splits. Here $\Facst{d}(\underline{n},R)$ is a subring of the product $R^n$ consisting of elements which are constant modulo \((d)_{\pfd}\) defined in \Cref{acst-notation}. See also another description of this ring in \Cref{split_ramification_is_point_other_form}.

\begin{thm}[\Cref{split_ramification_is_point}]\label{split_ramification_is_point_intro} Let $R$ be a perfectoid ring and $m(t)$ be a split monic polynomial of degree $n$.
Assume that $(R/dR)\to (R/dR)[t]/(m(t))$ induces an isomorphism at perfectoidization. Then if $A=R[t]/(m(t))$, we have
\[A_{\pfd}=\Facst{d}(\underline{n},R).\] 
\end{thm}

We give several examples of applications of \Cref{split_ramification_is_point} in \Cref{split_finite_examples}.

\subsection*{Acknowledgments}
This work was started in the problem session of the conference `\(p\)-adic and Characteristic \(p\) Methods in Algebraic Geometry' at EPFL. We are very grateful for these opportunities and their hospitality. The authors thank Marta Benozzo and Kęstutis Česnavičius for interesting discussions at this conference on the subject of the paper. The second-named author thanks Lucas Gerth, Alapan Mukhopadhyay and Zsolt Patakfalvi for helpful and interesting discussions on the subject of the article.
The first-named author was supported by JSPS KAKENHI Grant number 24KJ1085. The second-named author was supported by the FNS grant, \#200021-
231484.

\section{General study of perfectoidization of finite algebras over a perfectoid ring}\label{first_sec}

\subsection{Preliminaries}

\subsubsection{$p$-adic approximation of $p$-power roots of elements and perfectoidization.}\label{approx_section}

In characteristic \(p\), perfectoidization agrees with perfection, \emph{i.e}. with the colimit along Frobenius. Thus, starting from a ring \(A\), one obtains a new ring \(A_{\perf}\) in which every element becomes a \(p\)-power root of an element of \(A\).

In mixed characteristic, one can observe an analogous phenomenon, although only up to \(p\)-adic approximation. More precisely, elements of a perfectoidization, when they exist, can be viewed as \(p\)-adic approximations to compatible systems of \(p\)-power roots of elements of the original ring. \Cref{approxppowerpfd} gives a precise formulation of this principle.

In \Cref{approx_section}, we analyze this approximation phenomenon in detail. Afterward we will use it to prove \Cref{prootclosed_fetwayfromp}, which gives a mixed-characteristic analogue of the fact that, for semiperfectoid rings, perfectoidization can be described in terms of \(p\)-root closure; see, for example, \cite[Theorem 3.9]{struc_pfd} and \cite[Theorem 1.3]{ishizuka2024Calculationa}.


First, note the following aspect of perfectoid rings.

\begin{lemma}\label{pradform}Let $R$ be a perfectoid ring. Let $\varpi\in R$ be an element admitting a compatible sequence of $p$-power roots and such that $\varpi=pu$ for a unit $u\in R$. Then
\[\sqrt{pR}=\bigcup_{n\in \N}\varpi^{1/p^n}R\]
    
\end{lemma}
\begin{proof}
    We first show the claim when $R$ is $p$-torsion-free. Let $y\in \sqrt{pR}$. Then there is some integer $n\in \N$ such that
    \[y^{p^n}=px=(\varpi^{1/p^n})^{p^n}x'\]
    for some $x,x'\in R$. In particular in $R[\frac{1}{p}]$ we have 
    \[\left(\frac{y}{\varpi^{1/p^n}}\right)^{p^n}\in R.\]
    As $R\subseteq R[\frac{1}{p}]$ is $p$-root closed because $R$ is perfectoid, we deduce that $y\in \varpi^{1/p^n}R$, concluding.

    Now we tackle the general case. By \cite[Section 2.1.3]{cesnavicius2023purityflatcohomology} we know that $R$ is a subring of a product $R'\times S$ where $R'$ is a $p$-torsion-free perfectoid and $S$ is a perfect ring in characteristic $p$. Seeing $R$ has a subring of such a product and taking $r=(r',s)$ an element in $\sqrt{pR}$, this means that there some power such that $(r'^{N},s^N)=(px,0)$ for some $x\in R'$. As $S$ is reduced, we get $s=0$ and $r'\in \bigcup_{n\in \N}\varpi^{1/p^n}R'$, concluding.
\end{proof}

\begin{proposition}\label{approxppowerpfd} Let $R$ be a perfectoid ring and let $R\to A$ be an $R$-algebra such that there exists an initial map $A\to A_{\pfd}$ to a perfectoid ring. Then given $y\in A_{\pfd}$ and an integer $n\geq 1$, there exists $m\in \N$ and $x\in A$ such that

\[y^{p^m}\in x+p^nA_{\pfd}.\]
    
\end{proposition}

\begin{proof}Because 
\[A\to A_{\pfd}/\sqrt{pA_{\pfd}}\]
is initial among maps from $A$ to perfect rings in characteristic $p$, we deduce that the canonical map
\begin{equation}\label{comesfromA}
    (A/pA)_{\perf}\to A_{\pfd}/\sqrt{pA_{\pfd}}
\end{equation}

is an isomorphism. Let $y\in A_{\pfd}$. By \Cref{comesfromA} and \Cref{pradform} there is some $k,l\in \N$ and $x_0\in A$ such that

\[y^{p^k}\in x_0+\varpi^{1/p^l}A_{\pfd}.\]
So in the above equation, $x_0$ is to be understood as the image of $x_0\in A$ by the map $A\to A_{\pfd}$.

We will use the following algebraic fact.
\begin{claim}Let $R\to R'$ be any $R$-algebra and $a,b\in R'$. Let $i\geq 1$. Then if
\[a\equiv b \mod \varpi^{1/p^i}R'\]
Then we have
\[a^p\equiv b^p \mod \varpi^{1/p^{i-1}}R'\]
\end{claim}
\begin{claimproof}
Let us prove the claim -- denote $\sigma=\varpi^{1/p^{i}}$, so that $\sigma^{p^i}R'=pR'$. For an ideal $I\subseteq R'$, it holds by the binomial formula that 
\begin{equation*}
    a\equiv b \mod I \implies a^p \equiv b^p \mod pI+I^p.
\end{equation*}
Therefore for $I=\sigma R'$ we get that
\begin{equation*}
     a^p \equiv b^p \mod (p\sigma R'+\sigma^pR')=(\sigma^{p^i+1}R'+\sigma^pR')=\sigma^pR'
\end{equation*}
because $p^i+1\geq p$ for any prime $p$ and $i\geq 1$. This proves our claim. 
\end{claimproof}

Coming back to the proof of the proposition, it now follows by induction using the above claim that
\[y^{p^{k+l}}\in x_0^{p^l}+pA_{\pfd}.\]

Now, we can use another identity, which is a thematic variation of the claim above.
Namely, for any ring $R'$ and $a,b\in R'$, we have,
\[ a\equiv b \mod pR' \implies a^{p^{n}}\equiv b^{p^n} \mod p^{n+1}R'. \]
This is proved similarly using the binomial formula. Therefore it now follows that
\[y^{p^{k+l+n-1}}\in x_0^{p^{l+n-1}}+p^nA_{\pfd},\]
which concludes.
\end{proof}
\begin{rem}The fact used in the proof that for any ring $R'$ and $a,b\in R'$, we have, 
\[ a\equiv b \mod pR' \implies a^{p^{n}}\equiv b^{p^n} \mod p^{n+1}R' \] can be seen as the crucial algebraic reason for which
the untilt map $\limfro R'/p R'\to R'$ sending $(\overline{r}_n)\mapsto \lim_{n\in \N}r_n^{p^n}$ is well defined when $R'$ is a $p$-complete ring. It is arguable that this is one of the core engine of the theory of perfectoid rings.
\end{rem}

\subsubsection{Some results about torsion}

\Cref{nzd_syntomic_sections} has a reassuring application, see \Cref{nzd_reassuring}, in the context of perfectoidization of finite algebras over a perfectoid ring.
\begin{proposition}\label{nzd_syntomic_sections} Let $R$ be a perfectoid ring. Let $R\to A$ be a finite and finitely presented $R$-algebra which has bounded $p^{\infty}$-torsion. Suppose that there is a quasisyntomic and $p$-completely faithfully flat cover $R\to R_1$ such that the derived $p$-complete base change $A_1 \defeq A \widehat{\otimes}^L_R R_1$ of $A$ to $R_1$ admits a finite number of sections $(s_i\colon A_1\to R_1)_{i=1,\dots, k}$ as ring homomorphisms such that the induced map $A_1\to \prod_{i=1}^k R_1$ is surjective on $\Spec$. 

Then for any non-zero divisor $d\in R$ such that $R/dR$ is of bounded $p^{\infty}$-torsion, $d$ is also a non-zero divisor in $A_{\pfd}$.
    
\end{proposition}

\begin{proof}
    Note that under the above hypothesis, the map
    \[A_1\to \prod_{i=1}^k R_1\]
    is finite because it factorizes the diagonal map $R_1\to A_1\to \prod_{i=1}^k R_1$. As it is finite and surjective on $\Spec$, this is in particular an arc cover.

    Now we use \cite[Proposition 7.11]{bhatt2022Prismsa}. Namely, we find a \(p\)-completely faithfully flat perfectoid \(R_1\)-algebra $R_1\to R'$ with $R\to R'$ which is $p$-completely faithfully flat. Set $A' \defeq A \widehat{\otimes}^L_R R' \cong A_1 \widehat{\otimes}^L_{R_1} R'$ the derived $p$-complete base change of $A$ to $R'$, which is discrete by \cite[Lemma 5.1]{struc_pfd}. 
    Base changing the above arc cover, we have a \(p\)-complete arc cover
    \[A\to A'\to \prod_{i=1}^k R'\]
    since \(A \to A'\) is \(p\)-completely faithfully flat.
    Using \cite[Lemma 2.3]{struc_pfd} we get an injective map
    \[A_{\pfd}\to \prod_{i=1}^k R'.\]
    Now under our assumptions, using \cite[Lemma 5.2]{struc_pfd}, \(d\) becomes a non-zero divisor on \(R'\) and we can conclude.
\end{proof}

\begin{cor}\label{syntomic_nzd_application}Let $R$ be perfectoid and $R\to A$ be a finite and finitely presented map of one of these forms.
\begin{enumerate}
    \item The algebra is finite free monogenic, namely of the form $R\to A=R[t]/(m(t))$ where $m(t)\in R[t]$ is a monic polynomial.
    \item The map $R\to A$ is a finite syntomic cover and there is a finite group $G$ acting on $A$, such that the action map $A\otimes_R A\to \prod_{G}A$ surjective on $\Spec$.
\end{enumerate}
Then \Cref{nzd_syntomic_sections} applies.
\end{cor}
\begin{proof}
    For the first item, using that after a syntonic finite free base change, see \cite[\href{https://stacks.math.columbia.edu/tag/03HS}{Tag 03HS}]{stacks-project} and \cite[\href{https://stacks.math.columbia.edu/tag/00SR}{Tag 00SR}]{stacks-project}, the polynomial $m(t)$ splits, it suffices to show that the evaluation map at all roots is surjective on $\Spec$. Suppose that $m(t)$ splits already in $R$ for notational simplicity.
    We claim that the evaluation map is surjective with nilpotent kernel. Indeed, take any prime $\pk$ of $R$ and consider the base change
    \[(R/\pk)[t]/(m(t))\to R/\pk\times \cdots \times R/\pk.\]
Now, we can use the usual Euclidean division argument to see that an element in the kernel of the evaluation map is divisible by the product of $(t-\alpha_i)$'s where the $\alpha_i$'s ranges over \emph{distinct} values of the roots of $m(t)$ in the domain $R/\pk$.
This implies that an element of the kernel of the evaluation map is nilpotent in $(R/\pk)[t]/(m(t))$. 
So this element belongs to every prime $\qk$ of $A$ that contains $\pk A$. But as $R\to A$ is surjective on $\Spec$, we see that an element of the kernel of the above map as to be in the intersection of all primes of $A$.

As for the second item we can take $R=R_1$ directly, following notations of \Cref{nzd_syntomic_sections}.
\end{proof}

\begin{rem}\label{nzd_reassuring} Let $R$ be a perfectoid ring and $m(t)$ be a monic polynomial in $R[t]$. Suppose that the discriminant (see \Cref{disc_def}) of $m(t)$ is an element $d$ which a non-zero divisor such that $R/dR$ has bounded $p^{\infty}$-torsion. Let $A=R[t]/(m(t))$. Then using \Cref{syntomic_nzd_application}, we get that $A_{\pfd}$ is $d$-torsion-free. As a consequence, using that $A\to A_{\pfd}$ is a $d$-isogeny (\cite[Theorem 10.9]{bhatt2022Prismsa}) the map
\[A_{\pfd}\to A_{\pfd}\left[\frac{1}{d}\right]=A\left[\frac{1}{d}\right]\]
is injective. 

Note that using for a general $A$ which is finite and finitely presented over $R$ and étale away from $d$, then the kernel of $A_{\pfd}\to A\left[\frac{1}{d}\right]$ is anyway really tame as it is $(d)_{\pfd}$-almost zero using \cite[Theorem 1.1]{struc_pfd}.
\end{rem}

\subsection{A finiteness property of perfectoidization of finite algebras over perfectoid rings}
Say $R\to A$ is a finite and finitely presented map. Then, even if $R\to A_{\pfd}$ is far from being finite, there is some \say{hidden finiteness} of the perfectoidization. Namely, it is a consequence of \cite[Proposition 2.4.7]{cai2023Perfectoid} that we recall below.

\begin{proposition}[{cf.~\cite[Proposition 2.4.7]{cai2023Perfectoid}}]
\label{hidden_fin} Let $R$ be a perfectoid ring and $R\to A$ be a finite and finitely presented map to a ring which has bounded $p^{\infty}$-torsion. Suppose that this map is étale away from a finitely generated ideal $J$. Then for each $g\in J$ there is a $N\in \N$ such that $\Cone(A\to A_{\pfd})$ is killed by $g^N$. In particular as $J$ is finitely generated, there is $n_0\in \N$ with
\[J^{n_0}A_{\pfd}\subseteq \ima(A\to A_{\pfd})\]
    
\end{proposition}

Note that when $J=(p)$, there is a simple proof of the above which does not require André's Lemma.

\begin{proposition}\label{boundedptors} 

Let $R$ be a $p$-torsion free perfectoid ring and let $R\to A$ be a finitely presented finite map to a $p$-torsion free ring \(A\) such that $R[\frac{1}{p}]\to A[\frac{1}{p}]$ is finite étale.
Then there is $n_0\in \N$ with
\[p^{n_0}A_{\pfd}\subseteq A.\]

\end{proposition}

\begin{proof}Let us evacuate any potential imprecision from the last sentence of the statement: recall that because $A$ is supposed to be $p$-torsion free, $A_{\pfd}$ is also by \cite[Lemma A.2]{ma2022Analogue} -- see also \cite[Corollary 2.5]{struc_pfd}. Therefore by \cite[Theorem 10.9]{bhatt2022Prismsa} which implies that $A\to A_{\pfd}$ is an isomorphism when inverting $p$, the rings $A$ and $A_{\pfd}$ can be seen as subrings of $A[\frac{1}{p}]$ which are equal when inverting $p$. 


Now $A$ and $A_{\pfd}$ are both (classically) derived $p$-complete: $A$ is derived $p$-complete because it is a finitely presented $R$-module, and $R$ is (classicaly) derived $p$-complete because it is perfectoid -- same goes for $A_{\pfd}$. Now using \Cref{boundedtorsion_iso_awayfromp_completemodules} for $A\subseteq A_{\pfd}$, this concludes.
\end{proof}

\begin{lemma}[{cf.~\cite[Theorem 2.3]{bhatt2019Torsion}}]\label{boundedtorsion_iso_awayfromp_completemodules} Let $M'\subseteq M$ be an inclusion of derived $p$-complete abelian groups. If $M'[\frac{1}{p}]=M[\frac{1}{p}]$, then there is a $n_0\in \N$ such that $p^{n_0}M\subseteq M'$. Note in particular that the subspace topology on $M'$ for the $p$-adic topology on $M$ agrees with the $p$-adic topology on $M'$.
    
\end{lemma}
\begin{proof}Because $M/M'$ is a $p$-torsion derived $p$-complete module, the result follows from \cite[\href{https://stacks.math.columbia.edu/tag/0CQY}{Tag 0CQY}]{stacks-project} which uses Baire category's theorem \cite[\href{https://stacks.math.columbia.edu/tag/0CQU}{Tag 0CQU}]{stacks-project} as a key input.
    
\end{proof}

\subsubsection{Monogenic case}

In this section we improve \Cref{hidden_fin} in the special monogenic case $A=R[t]/(m(t))$ where $m(t)$ is a monic polynomial. Namely, we show that the discriminant of the extension (\Cref{disc_def}) suffices and provides a control on the $n_0\in \N$ from \Cref{hidden_fin}. The proof of \Cref{boundeddtorsmonic} is simpler but goes along the same lines as the one in \cite[Proposition 2.4.7]{cai2023Perfectoid}.

\begin{rem}\label{disc_def}In what follows we consider the quantity,
\[d\defeq \Res(m(t),m'(t))\in R \]
    the \emph{resultant} of the polynomials $m(t)$ and $m'(t)$. The resultant of a given polynomial \(f(t)\) and its derivation \(f'(t)\) is called the \emph{discriminant} of \(f(t)\).\footnote{Sometimes another sign convention is taken.} If \(R\) is an integral domain, \(\Res(f(t), f'(t))\) becomes zero exactly when \(f(t)\) has a multiple root.
In general, for two polynomials $f(t)$ and $g(t)$ with coefficient in $B[t]$ of degrees $n$ and $m$ respectively where $B$ is any ring, the resultant \(\Res(f(t), g(t))\) is defined as the determinant of the map of free $B$-modules of rank $n+m$
\begin{align*}
    B[t]_{<n}\oplus B[t]_{<m}\to B[t]_{<n+m} \\
    (a(t),b(t))\mapsto a(t)g(t)+b(t)f(t)
\end{align*}
where the symbol \((-)_{< n}\) means that the \(B\)-submodules of \(B[t]\) consisting of polynomials whose degree is strictly less than \(n\).
One verifies that
\begin{enumerate}
    \item For any $b\in B$, we have $\Res(t-b,g(t))=g(b)$.
    \item If $f(t)=f_1(t)f_2(t)\in B[t]$, then $\Res(f(t),g(t))=\Res(f_1(t),g(t))\Res(f_2(t),g(t))$.
\end{enumerate}
Note that the resultant is functorial in ring maps which preserves the degree of $f$ and $g$. Therefore, to prove the above one can suppose that $B$ is a domain, namely a polynomial $\Z$-algebra where the indeterminates are the coefficients of $f$ and $g$. Then one can reduce the arguments to the case where $B$ is an algebraically closed field where it follows from \cite[Chapter IV, Section 8]{Lang2002_Algebra}. 

If $m(t)$ splits, then the evaluation map at all roots of $m(t)$ is injective when $d$ is a non-zero divisor and is an isomorphism when $d$ is invertible. Therefore, in general, as the property of being étale can be checked after faithfully flat base change, $R\to R[t]/m(t)$ is finite étale away from $d$.
    
\end{rem}

We now refine \Cref{hidden_fin} in the monogenic case.

\begin{thm}\label{boundeddtorsmonic}
Let $R$ be a perfectoid ring, $m(t)\in R[t]$ be a monic polynomial. Let $A=R[t]/m(t)$ and $$d\defeq \Res(m(t),m'(t))\in R.$$
Then $\Cone(A\to A_{\pfd})$ is killed by $d$.

In particular if $d$ is a non-zero divisor in $R$\footnote{This condition implies that \(m(t)\) can not be written in any faithfully flat extension has a split polynomial with multiple roots.} and that $R/dR$ is of bounded $p^{\infty}$-torsion, then $A_{\pfd}$ is $d$-torsion free and a subring of $A[1/d]$ (\Cref{nzd_reassuring}), in which case we have
\[dA_{\pfd}\subseteq A\]
\end{thm}

\begin{proof} 
    Let $R\to R'$ be a $p$-completely faithfully flat extension where $m(t)$ splits in $R'$. For example, one could use André's lemma \cite[Theorem 7.14]{bhatt2022Prismsa}. Denote by $A'$ the derived $p$-complete derived base change of $A$ to $R'$. As $A$ is a finite free $R$-module and that $R$ is of bounded $p^{\infty}$-torsion, this base change is discrete by \cite[Lemma 5.1 (1)]{struc_pfd}.

  In this case, the monic polynomial \(m(t)\) splits in \(R'\) and we can consider an evaluation map
    \[A'\to R'\times \cdots \times R'\]
This map is a map between free modules of the same rank that can be represented by the Vandermonde matrix $V$ of the roots. We have that $d=\det(V)^2$ up to sign. Therefore it follows that the kernel and cokernel of this map is killed by $\det(V)$ by Cramer's rule. Now using universal property of perfectoidization we have a factorization
\[A'\to A'_{\pfd}\to R'\times \cdots \times R'\]
   where the second map is injective because the evaluation map is an arc-cover by the proof of \Cref{syntomic_nzd_application} and \cite[Lemma 2.3]{struc_pfd} -- therefore the above argument implies that the kernel and cokernel of $A'\to A'_{\pfd}$ is killed by $\det(V)$. Therefore by \cite[Lemma 3.2]{Bhatt_2012}, $\Cone(A'\to A'_{\pfd})$ is killed by $d=\det(V)^2$. By the first part of the proof of \cite[Proposition 2.4.7]{cai2023Perfectoid} it implies the claim before $p$-completely faithfully flat base change to $R'$.
\end{proof}

\begin{example}\label{not_syntomic} We see the behavior of \Cref{boundedptors} in the following example. We note that this may also hint that \Cref{boundeddtorsmonic} may generalize to finite locally free algebras, where the discriminant is defined in general (\cite[\href{https://stacks.math.columbia.edu/tag/0BVH}{Tag 0BVH}]{stacks-project}) as the example below is not even syntomic.
Let $R$ be a $p$-torsion free ring. Consider 
  \[A\defeq \{(r_1,r_2,r_3)\subseteq R\times R\times R \mid r_1\equiv r_2\equiv r_3 \mod p \}\]
  We claim that it's a finite free module of rank $3$ over $R$. 
  Namely, any element is of the form $(r,r+p\alpha,r+p\beta)$ for unique $(r,\alpha,\beta)\in R^3$. Therefore we see that $(1,1,1),(0,p,0),(0,0,p)$ form an $R$-basis for example. We claim that there is a presentation of the form
  \[R[x,y]/(xy,x^2-px,y^2-py).\]
  Note that this finite and faithfully flat extension is not syntomic because it is not modulo $p$. The discriminant can be computed from the following observations: $\Tr_{A / R}(y)=\Tr_{A / R}(x)=p$ and $\Tr_{A / R}(x^2)=\Tr_{A / R}(y^2)=p^2$: the matrix of the trace is therefore
  \[ \begin{pmatrix}
3 & p & p \\
p & p^2 & 0 \\
p & 0 & p^2 
\end{pmatrix}  \]
We deduce that discriminant ideal is $(p^4)$. Note that inverting $p$ turns the defining inclusion $A\subseteq R^3 $ to an isomorphism.

If $R$ is perfectoid, using \cite[Proposition 2.7]{struc_pfd}, \cite[Corollary 8.12]{bhatt2022Prismsa} we have
\[\begin{tikzcd}
A_{\pfd} \arrow[d] \arrow[r] & R\times R\times R \arrow[d]         \\
(A/pA=R/pR)_{\pfd} \arrow[r] & (R/pR\times R/pR\times R/pR)_{\pfd}
\end{tikzcd}\]
And therefore:
\[A_{\pfd}=\{(r_1,r_2,r_3) \mid r_1\equiv r_2\equiv r_3 \mod \sqrt{(p)}\}\]
Note that 
\[pA_{\pfd}\subseteq A.\]
Therefore it can happen that there is an element which divides the discriminant with this property already: this shows that the behavior in \Cref{boundeddtorsmonic} is not ``sharp''. Note also that $x=(p^{p^{\frac{p-1}{p}}},0,0)$ is in $A^{\pfd}$. But $p^{1/p}x=(p^{1/p},0,0)$ which is not in $A$. Therefore we see that $A\to A_{\pfd}$ is not a $p$-almost isomorphism.
    
\end{example}

\subsection{$p$-root closure and perfectoidization for finite maps}

As a consequence of \Cref{boundedptors} and \Cref{approxppowerpfd}, we get the following nice result, which echoes \cite[Theorem 3.9]{struc_pfd}.

The same result appears in \cite[Remark 3.5.4(2)]{bhatt2026Aspects} with a slightly different approach.

\begin{proposition}
    
\label{prootclosed_fetwayfromp} Let $R$ be a $p$-torsion free perfectoid ring and let $R\to A$ be a finitely presented finite map to a $p$-torsion free ring such that the Zariski localization $R_p\to A_p$ is finite étale. Then $A_{\pfd}$ is the $p$-root closure of $A$ in $A[\frac{1}{p}]$. More precisely, it holds that 
\[A_{\pfd}=\{x\in A[\frac{1}{p}] \mid \exists m\in \N \quad x^{p^m}\in A\}.\]
\begin{proof} 
    Let $n_0\in \N$ be an integer such that
    \[p^{n_0}A_{\pfd}\subseteq A,\]
    which exists by \Cref{boundedptors}. Now take any element $y\in A_{\pfd}$. By \Cref{approxppowerpfd}, there is some $x\in A$ and $m\in \N$ such that
    \[y^{p^m}\in x+p^{n_0}A_{\pfd}\subseteq A.\]
    Therefore we see that $A_{\pfd}$ is a subring of the $p$-root closure of $A$ in $A[\frac{1}{p}]$. But this was actually the non trivial inclusion: indeed using the fact that $A_{\pfd}$ is $p$-root closed in $A_{\pfd}[\frac{1}{p}]$ because it is perfectoid, we see that it immediately implies that the $p$-root closure of $A$ in $A[\frac{1}{p}]$ is a subring of $A_{\pfd}$. The last part of the statement follows from \cite[Proposition 1]{roberts2008Root}.
\end{proof}
    
\end{proposition}
\begin{rem}
    It is highly noticeable that no $p$-completion is needed in \Cref{prootclosed_fetwayfromp}. This phenomenon is in contrast from the case of semiperfectoid rings, see \cite[Remark 3.11]{struc_pfd}.
\end{rem}

\subsection{Density principles}

In this section, we explain that in good cases, to get from $A$ to $A_{\pfd}$ when $R\to A$ is finite and finitely presented map, amounts to adding $p$-roots of elements modulo $p$.


\begin{thm}\label{densityprinciple}
Say $R$ perfectoid ring
and $d\in R$ is a non-zero divisor.
Say that $R\to A$ is a finite and finitely presented map étale outside of $V(d)$, such that $d$ is also a non-zero divisor in $A$.
In this case, $A$ can be seen as a subring of $A_{\pfd}$.
Take an intermediate subring
\begin{equation*}
    A \subseteq C \subseteq A_{\pfd}
\end{equation*}
such that \(C/pC\) is semiperfect.

Then \(C\) is dense in \(A_{\pfd}\) for the subspace topology where $A_{\pfd}$ is equipped with the $p$-adic topology -- more precisely the natural morphism
 \[\widehat{C} \to  A_{\pfd}\] 
 is surjective, where \(\widehat{C}\) is the \(p\)-adic completion of \(C\).


\end{thm}
\begin{proof}
    The fact that $A\subseteq A_{\pfd}$ follows from the assumption that $d$ is a non-zero divisor in $A$ and that this map is $d$-isogeny \cite[Theorem 10.9]{bhatt2022Prismsa}.
    
    Since \(\widehat{C}\) is a semiperfectoid ring and the map \(\widehat{C} \to A_{\pfd}\) is a map between \(p\)-complete rings, the image \(D \defeq \Image(\widehat{C} \to A_{\pfd})\) is also $p$-complete and semiperfectoid and is contained in the topological closure \(\overline{C}\) of \(C\) in \(A_{\pfd}\).
    However $D \subseteq A_{\pfd}$ implies that $D$ is perfectoid by \cite[Lemma 2.3]{Dine_2024} (using \cite[Theorem 7.4]{bhatt2022Prismsa}). Because $D \subseteq A_{\pfd}$ factorizes the universal inclusion $A \subseteq A_{\pfd}$ and is perfectoid, it follows that $D=A_{\pfd}$ and especially \(\overline{C} = A_{\pfd}\) and the surjectivity \(\widehat{C} \twoheadrightarrow A_{\pfd}\) holds.

 \end{proof}

 \begin{cor}\label{densityprincipleVp}In the setup of \Cref{densityprinciple}, and when $V(d)=V(p)$ topologically, we can conclude that 
\[\widehat{C}=\overline{C}=A_{\pfd}.\]

 \end{cor}
\begin{proof}

From \Cref{boundedptors}, we see that $A_{\pfd}/C$ is of bounded $p^{\infty}$-torsion.
Now, if we denote derived $p$-completion by $\Lambda_p$, and apply it to the exact sequence
\[0\to C\to A_{\pfd}\to A_{\pfd}/C\to 0\]
we get a fiber sequence of derived $p$-complete modules
\[\widehat{C}\to A_{\pfd}\to \Lambda_p \left( A_{\pfd}/C\right) \]
But because $A_{\pfd}/C$ is of bounded $p^{\infty}$-torsion, we have $\pi_1(\Lambda_p\left( A_{\pfd}/C\right) )=0$, implying in the end that 
\[\widehat{C}\to A_{\pfd}\]
is injective as wanted.

\end{proof}



\section{Computations}\label{sec_ex}

\subsection{Adding $m$-th roots.}\label{mth_roots_section}
In this subsection, we concentrate our attention to the case where $A=R[t]/(t^m-r)$ where $(m,p)=1$.

\begin{thm}
    
\label{ComputePerfectoidizationTm} 
    Let \(R\) be a perfectoid ring. Let $m\in \N$ be an integer such that $(p,m)=1$. Let \(R\abracket{t^{1/p^\infty}}\) be the \(p\)-adic completion of \(\cup_{n \geq 0} R[t^{1/p^n}]\).
    Then the perfectoidization 
    \begin{equation*}
        R\abracket{t^{1/p^\infty}}[t^{1/m}]\to \left(R\abracket{t^{1/p^\infty}}[t^{1/m}]\right)_{\pfd}
    \end{equation*}
    identifies with the natural inclusion  
    \begin{equation*}
        R\abracket{t^{1/p^\infty}}[t^{1/m}]\to R\abracket{t^{1/mp^\infty}}
    \end{equation*}
    where $R\abracket{t^{1/mp^\infty}}$ is the \(p\)-adic completion of \(\bigcup_{n \geq 0} R[t^{1/mp^n}].\)

\end{thm}

\begin{proof}
  Let $k_0$ be the multiplicative order of $p$ in $(\Z/m\Z)^{\times}$. Then for every $n\geq 0$ we have that $m \mid 1-p^{k_0n}$ in $\Z$. Write $l_n\defeq \frac{p^{nk_0}-1}{m}\in \N$. 


    We first argue that the canonical map
    \begin{equation}\label{WantIso}
        R\abracket{t^{1/p^\infty}}[t^{1/m}][1/t]\to R\abracket{t^{1/mp^\infty}}[1/t]
    \end{equation}
    is an isomorphism. We notice that the injectivity follows from the injectivity of $R\abracket{t^{1/p^\infty}}[t^{1/m}]\to R\abracket{t^{1/mp^\infty}}$ and that $t$ is a non-zero divisor in both those rings.
    
    Now, note for every $n\geq 0$ we have that
    \begin{equation}
        \label{PlayWithFractions}
        \frac{1}{mp^{nk_0}}=\frac{1}{m}+\frac{\frac{1-p^{nk_0}}{m}}{p^{nk_0}}=\frac{1}{m}-\frac{l_n}{p^{nk_0}}
    \end{equation}

Note also that $l_n< p^{nk_0}$ by construction. So by \Cref{PlayWithFractions}, for any $n\in \N$, we have
\begin{equation}\label{PlayWithExponents}
    t^{1/mp^{nk_0}}t=(t^{1/mp^{nk_0}}t^{l_n/p^{nk_0}})t^{(p^{nk_0}-l_n)/p^{nk_0}}=t^{1/m}t^{(p^{nk_0}-l_n)/p^{nk_0}}\in  R\abracket{t^{1/p^\infty}}[t^{1/m}].
\end{equation}
For $1\leq i \leq m-1$, the above \Cref{PlayWithExponents} implies that $t^{i/mp^{nk_0}}t^i\in R\abracket{t^{1/p^\infty}}[t^{1/m}]$. Also note that any power series $f\in R\abracket{t^{1/mp^\infty}}$ is a sum of monomnials of the form $q(t)=t^{a/mp^{nk_0}}$ for some $a,n\in \N$ -- and if we write $a=mq+r$ by Euclidean division with $q\in \N$ and $0\leq r<m$, we can rewrite this as $q(t)=t^{q/p^{nk_0}}t^{r/mp^{nk_0}}$. Therefore, we see that for every such monomial $q(t)$ we have $t^{m-1}q(t)\in R[t^{1/p^\infty}][t^{1/m}]$. Therefore, we also see that $t^{m-1}f\in R\abracket{t^{1/p^\infty}}[t^{1/m}]$.


Therefore, this implies that the map in \Cref{WantIso} is surjective. 
    
    We can now use \cite[Proposition 2.7]{struc_pfd} (which is \cite[Corollary 8.12]{bhatt2022Prismsa} explained in this context), to get the following pullback square
    \[\begin{tikzcd}
    \left(R\abracket{t^{1/p^\infty}}[t^{1/m}]\right)_{\pfd} \arrow[d] \arrow[r] & R\abracket{t^{1/mp^\infty}} \arrow[d]    \\
    \left(R\abracket{t^{1/p^\infty}}[t^{1/m}]/tR\abracket{t^{1/p^\infty}}[t^{1/m}]\right)_{\pfd} \arrow[r]      & \left(R\abracket{t^{1/mp^\infty}}/tR\abracket{t^{1/mp^\infty}}\right)_{\pfd}
    \end{tikzcd}\]
    Note that we used that $R\abracket{t^{1/mp^\infty}}$ is perfectoid, being abstractly isomorphic to $R\abracket{t^{1/p^\infty}}$ which is perfectoid (see \cite[Proposition 2.1.11 (a)]{cesnavicius2023purityflatcohomology} for example).
    
    
    Now, note that as perfectoid rings are $p$-complete (by definition) and  reduced \cite[Subsection 2.1.3]{cesnavicius2023purityflatcohomology}, perfectoidizations of rings in the lower part of the pullback square will necessarily factor through the quotients by the $p$-adic closure of $\sqrt{tR\abracket{t^{1/p^\infty}}[t^{1/m}]}$ and $\sqrt{tR\abracket{t^{1/mp^\infty}}}$ respectively. But those quotients identify to $R$, which is already perfectoid, implying that the perfectoidizations are those quotients using that the perfectoidization is the initial map to a perfectoid ring. As after those quotients the lower map in the pullback square identifies with \(R \xrightarrow{\id_R} R\), the proof of the lemma concludes by observing that pullbacks preserve isomorphisms.
\end{proof}

\begin{rem}
    In \cite[Example 2.3.12]{reinecke2020Moduli}, the authors study the perfectoidization of \(K^\circ\abracket{t^{1/p^\infty}}[t^{1/2}]\) for a perfectoid valuation ring \(K^{\circ}\) and proved that the perfectoidization is \((\varpi x)^{1/p^\infty}\)-almost isomorphic to \(K^\circ\abracket{t^{1/2p^\infty}}\) where \(\varpi\) is a pseudo-uniformizer of \(K^{\circ}\). The proof relies on the perfectoid version of the Riemann extension theorem developed in \cite[Theorem 4.2]{bhatt2018Direct} which requires the almost mathematics.
    However, as we proved above, using another technique relying on \cite[Proposition 2.7]{struc_pfd} (which explains why we can use \cite[Corollary 8.12]{bhatt2022Prismsa} in our case), we can show that the perfectoidization of \(K^\circ\abracket{t^{1/p^\infty}}[t^{1/2}]\) is isomorphic to \(K^\circ\abracket{t^{1/2p^\infty}}\) without using almost mathematics.
\end{rem}

\begin{rem}\label{discappears}In the proof of \Cref{ComputePerfectoidizationTm}, the discriminant $t^{m-1}$ of the extension was used. This is not a coincidence, the way it was utilized aligns with \Cref{boundeddtorsmonic}.
    
\end{rem}

The universality of \Cref{ComputePerfectoidizationTm} allows to get the following corollary.

\begin{cor} \label{ComputePerfectoidizationRootm}
    Let \(R\) be a perfectoid ring and let \(r\) be an element of \(R\) which has a compatible sequence of \(p\)-power roots \(\{r^{1/p^n}\}_{n \geq 0}\) in \(R\). Let $m\in \N$ be an integer such that $(m,p)=1$. 
    Then the perfectoidization of \(R[x]/(x^m-r)\) identifies to the natural map
    \[R[x]/(x^m-r)\to \left(R[x^{1/p^\infty}]/(x^{m/p^n}-r^{1/p^n})_{n \geq 0}\right)^{\wedge}\]
    where the right hand side denotes the \(p\)-adic completion of \(R[x^{1/p^\infty}]/(x^{m/p^n}-r^{1/p^n})_{n \geq 0}\).
\end{cor}
\begin{proof}
      Set \(S \defeq R\abracket{t^{1/p^\infty}}[x]/(x^m-t)\cong R\abracket{t^{1/p^\infty}}[t^{1/m}]\). Note that $S$ is $p$-completely faithfully flat over $R\abracket{t^{1/p^\infty}}$.
    Therefore,
    \begin{equation*}
        R[x]/(x^m-r) \cong R \widehat{\otimes}^L_{R\abracket{t^{1/p^\infty}}} S
    \end{equation*}
    where the map \(R\abracket{t^{1/p^\infty}} \to R\) is given by sending \(t\) to \(r\).
    By \Cref{ComputePerfectoidizationTm} above, the perfectoidization of \(S \cong R\abracket{t^{1/p^\infty}}[t^{1/m}]\) is isomorphic to \(R\abracket{t^{1/mp^\infty}}\), which is the \(p\)-adic completion of \(R\abracket{t^{1/p^\infty}}[x^{1/p^\infty}]/(x^{m/p^n} - t^{1/p^n})_{n \geq 0}\).
    Now, since the perfectoidization functor preserves pushouts of $E_{\infty}$-rings (\cite[Proposition 8.13]{bhatt2022Prismsa}), this shows the claim as
    \begin{equation*}
        R \widehat{\otimes}^L_{R\abracket{t^{1/p^\infty}}} S_{\pfd} \cong R \widehat{\otimes}^L_{R\abracket{t^{1/p^\infty}}} R\abracket{t^{1/mp^\infty}} \cong (R[x^{1/p^\infty}]/(x^{m/p^n}-r^{1/p^n})_{n \geq 0})^{\wedge_p}.
    \end{equation*}
    
\end{proof}

Under \Cref{ComputePerfectoidizationRootm}, we expect that the perfectoidization of finite extensions looks like ``adding all \(p\)-power roots of the added elements''. However, as below, the situation is not so optimistic, namely, adding (or sorting) all \(p\)-power roots does not give the perfectoidization map.

\begin{lemma} \label{Everything is about proots of unity somehow}
    Suppose that $p\neq 2$.
    Let $R$ be a perfectoid domain. Suppose that $1$ and $-1$ have all their $p$-power roots in $R$. Then the quotient map
    \begin{align*}
        R\abracket{t^{1/p^{\infty}},x^{1/p^{\infty}}}/(x^2-t) & \twoheadrightarrow R\abracket{t^{1/2p^{\infty}}} \\
        x^{1/p^n} & \mapsto t^{1/2p^n}
    \end{align*}
    is \emph{not} initial among map to perfectoid rings.\footnote{Note that as proved in \Cref{ComputePerfectoidizationTm}, the composition \(R\abracket{t^{1/p^\infty}}[x]/(x^2-t) \to R\abracket{t^{1/p^\infty}, x^{1/p^\infty}}/(x^2-t) \twoheadrightarrow R\abracket{t^{1/2p^\infty}}\) is initial among maps to perfectoid rings.}
\end{lemma}

\begin{proof}
    A map of $R$-algebras
    \begin{equation*}
        R\abracket{t^{1/p^{\infty}},x^{1/p^{\infty}}}/(x^2-t)\to R
    \end{equation*}
    consists of a pair of choices \((\{\alpha^{1/p^n},\beta^{1/p^n}\}_{n\geq 0})\) of elements in $R^{\flat}$ such that $\beta^2=\alpha$. Note that for $\beta=-1$ and $\alpha=1$ gives the existence of such a map by the assumption made on $R$. Fix such a map. 

    Note that such maps factors through the quotient \(R\abracket{t^{1/2p^\infty}}\) if and only if \((\beta^{1/p^n})^2=\alpha^{1/p^n}\) for all $n\geq 0$. Suppose by contradiction that the map in the statement of the Lemma is the perfectoidization. Then we have $(\beta^{1/p^n})^2=\alpha^{1/p^n}$ for all $n\geq 0$. 

    But now note that \((\{\alpha^{1/p^n},\xi_{p^n}\beta^{1/p^n}\})\) with $\xi_{p^n}$ a primitive $p^n$-root of unity also defines a map of $R$-algebras
    \begin{equation*}
        R\abracket{t^{1/p^{\infty}},x^{1/p^{\infty}}}/(x^2-t)\to R.
    \end{equation*}
    because the only condition is that the square of the zero-th term of the second compatible system of $p$-power roots is the zero-th term of the first compatible system of $p$-power root, a condition met by this two systems again (Note that $\xi_{p^0}=\xi_1=1$).
    As we suppose by contradiction that the map in the statement is the perfectoidization map, as $R$ is perfectoid, we would have a factorization, implying that \((\xi_{p^n}\beta^{1/p^n})^2=\alpha^{1/p^n}\).
    But then $\xi_{p^n}^{2}\alpha^{1/p^n}=\alpha^{1/p^n}$, implying that $\xi_{p^n}^{2}=1$, a contradiction.
\end{proof}



We can go further into the study of the perfectoidization of $R[t]/(t^m-d)$ in \Cref{Computation_pi_unit}, which implies a $p$-complete arc local description of the perfectoidization of such an algebra, see \Cref{arclocal_mth_roots}.

\begin{thm}
    \label{Computation_pi_unit}
    Let \(R\) be a \(\pi\)-torsion-free perfectoid ring with $D(\pi)=D(p)$ in $\Spec(R)$ admitting a compatible system \(\{\pi^{1/p^n}\}_{n \geq 0}\) of \(p\)-power roots of \(\pi\).
    Take a unit \(u\) in \(R\) and an integer \(m\) such that \((m, p) = 1\).
    Set
    \begin{equation*}
        A \defeq R[t]/(t^m-\pi u),
    \end{equation*}
    which is finite \'etale outside \(\pi\) over \(R\).
    Write the multiplicative order \(k_0\) of \(p\) in \((\setZ/m\setZ)^\times\) and let $q=p^{k_0}$.
    Then the perfectoidization of \(A\) can be identified with\footnote{The meaning of this union will be clarified in the proof.}
    \begin{equation*}
        A_{\pfd} = \parenlr{\bigcup_{n \geq 0} R\bracketlr{\parenlr{\pi^{1/q^n}u}^{1/m}}}^{\wedge_p}
    \end{equation*}
    where \((\pi ^{1/q^n}u)^{1/m}\) is a \(m\)-th root of \(\pi^{1/q^n}u\). 
\end{thm} 

\begin{proof}
    Using the almost purity theorem \cite[Theorem 10.9]{bhatt2022Prismsa}, the perfectoidization \(A \to A_{\pfd}\) is an isomorphism outside \(\pi\).
    Since \(A\) is \(\pi \)-torsion-free and \(u\) is a unit in \(R\), we see that \(A_{\pfd}\) is a $p$-root closed subring of \(A[1/\pi]=A[1/p]\) by \cite[Section 2.1.7]{cesnavicius2023purityflatcohomology}.
    As in \Cref{ComputePerfectoidizationTm}, we take the multiplicative order \(k_0\) of \(p\) in \((\setZ/m\setZ)^\times\) and a positive integer \(l_n \defeq (p^{nk_0}-1)/m\) since \(m \mid 1-p^{k_0n}\) for all \(n \geq 0\). We also write $q=p^{k_0}$, so that $l_n=(q^n-1)/m$.

Note the following: for $n\in \N$, if $\alpha_{n+1}$ is an element of an $R$-algebra with $\alpha_{n+1}^m=\pi^{1/q^{n+1}}u$, then $\beta\defeq \alpha_{n+1} \cdot (\pi^{1/q^{n+1}})^{\frac{q-1}{m}}$ is an element of this $R$-algebra with 
$$\beta^m=\pi^{1/q^{n+1}}u(\pi^{1/q^{n+1}})^{q-1}=\pi^{1/q^n}u.$$
Therefore, we can consider the direct colimit of $R$-algebras $B\defeq \varinjlim_{n\in \N}A_n$ with $A_0=A$ and for $n\in \N$ we define the transition map as: 
\[A_n=R[t]/(t^m-\pi^{1/q^n}u)\xrightarrow{t\mapsto t(\pi^{1/q^{n+1}})^{\frac{q-1}{m}}}A_{n+1}=R[t]/(t^m-\pi^{1/q^{n+1}}u).\]
We denote by $(\pi^{1/q^n}u)^{1/m}$ the image of $t$ at level $n$ in $B$.
We see that by construction that $A[\frac{1}{\pi}]=B[\frac{1}{\pi}]$. In particular, as $R$ is assumed to be $p$-torsion free, and $A_n$ is a finite flat $R$-algebra for all $n\in \N$, we see that each of the transitions map is injective, justifying the union in the statement.


Using \Cref{PlayWithFractions}, we can compute
      \begin{equation*}
        \left((\pi^{1/q^n}u)^{1/m}\right)^{q^{n}} = (\pi u)^{1/m}u^{l_n} \in A.
    \end{equation*}

This shows, using that $A_{\pfd}$ is $p$-root closed in $A[\frac{1}{\pi}]=B[\frac{1}{\pi}]$ that $B\subseteq A_{\pfd}$ is a sub-ring. We now show that $B/pB$ is semiperfect. It suffices to show that $(\pi^{1/q^n}u)^{1/m}$ can be written as a $p$-power of some element of $B$ up to an element of $pB$. Since $u^{l_1}$ is a unit in $R/pR$ and $R/pR$ is semiperfect, pick some $v\in R$ with $v^qu^{l_1}\equiv 1 \mod pR$. Therefore we can compute
\[\left((\pi^{1/q^{n+1}}u)^{1/m}v\right)^{q} =(\pi^{1/q^n} u)^{1/m}u^{l_1}v^q\equiv (\pi^{1/q^n} u)^{1/m} \mod pA_{n}. \]
We can now conclude that $B^{\wedge_p}=A_{\pfd}$ by \Cref{densityprincipleVp}.

\end{proof}

\begin{rem}\label{arclocal_mth_roots}
    Note that in any perfectoid valuation ring of rank 1 $R$, any non-invertible element is of the form $\pi u$ where $\pi$ is an element admitting a compatible sequence of $p$-th roots with $R[1/p]=R[1/\pi]$ and $u$ is a unit. Therefore \Cref{Computation_pi_unit} implies that we understand locally in the arc topology the perfectoidization of $R[t]/(t^m-d)$ where $R$ is a perfectoid ring and $(m,p)=1$. Indeed, take a map $R\to V$ to a perfectoid valuation ring of rank 1 (not necessarily absolutely integrally closed; in this case the results in \Cref{split_finite_examples} also give an answer). Then there are three cases.
    \begin{enumerate}
    \item $d$ is sent to zero and then the perfectoidization is $V$, the reduction of $V[t]/(t^m)$.
        \item $d$ is sent to a unit in $V$ and then the base change $V[t]/(t^m-d)$ is étale and therefore perfectoid. Note that in this case, as $d$ is invertible, the elements appearing in the construction of \Cref{Computation_pi_unit} are already in this algebra.
        \item $d$ is sent to a non zero and non-unit element, which is covered by \Cref{Computation_pi_unit}.
    \end{enumerate}
\end{rem}

\subsection{Split finite extensions}\label{split_finite_examples}


In this section, we treat some cases of split polynomials. The main tool to treat these examples is the following principle, highlighted in \Cref{split_ramification_is_point}. 

\begin{nota}\label{acst-notation}
    We introduce the following bit of notation. If $X$ is a set and $R$ is a perfectoid ring, we write
\[\Facst{d}(X,R)\]
for the subset of set theoretic functions from $X$ to $R$ which are constant modulo $(d)_{\pfd}$. This forms a ring using pointwise operations in $R$. Also we write $\underline{n}=\{1,\dots,n\}\subseteq \N$.
\end{nota}

\begin{thm}\label{split_ramification_is_point} Let $R$ be a perfectoid ring and $m(t)$ be a split monic polynomial of degree $n$ with roots $\alpha_1,\dots,\alpha_n$. Let $d\in R$ be the discriminant (\Cref{disc_def}) of $m(t)$, and assume it is a non-zero divisor in $R$. Assume that $(R/dR)\to (R/dR)[t]/(m(t))$ induces an isomorphism at perfectoidization. Then if $A=R[t]/(m(t))$, one can realize
\[A_{\pfd}=\Facst{d}(\underline{n},R)\]
    and an universal map $A\to A_{\pfd}$ is given by sending $t$ to the the evaluation to the map sending $i\mapsto \alpha_i$.
\end{thm}
\begin{proof}
    Using facts recalled in \Cref{disc_def} we see that the evaluation map at roots
    \[A\to R^n=\Fun(\{1,\dots,n\},R)\]
    is an isomorphism away from $d$. We also know that it is injective ($d$ is a non-zero divisor) and finite, hence we can apply \cite[Proposition 2.7]{struc_pfd} to get a pullback square
    \[\begin{tikzcd}
A_{\pfd} \arrow[r] \arrow[d]   & {\Fun(\{1,\dots,n\},R)} \arrow[d]  \\
(A/dA)_{\pfd} \arrow[r] & {\Fun(\{1,\dots,n\},R/(d)_{\pfd})}
\end{tikzcd}\]

Using our hypothesis that $(R/dR)\to (A/dA)$ is an isomorphism at perfectoidization, we see that the below horizontal map in the square above identifies with the structural $R/(d)_{\pfd}$-alegbra map $R/(d)_{\pfd}\to (R/(d)_{\pfd})^n$ which is the diagonal map, or the the constant functions if seen in the functional point of view.

The claim now follows from the usual description of pullbacks.
\end{proof}

\begin{rem}\label{split_ramification_is_point_other_form} In the setup of \Cref{split_ramification_is_point} let $f\in A_{\pfd}$. So there is a constant function, that we denote by $r\in R$ with
\[f=r+\epsilon\]
where $\epsilon \in \Fun(\{1,\dots,n\},(d)_{\pfd})$. 
We simply note that therefore another way to write this algebra is as follows, as a sub-algebra of the product $R^n$:
\[R_{\Delta}+(d)_{\pfd}\times \cdots \times (d)_{\pfd},\]
where $R_{\Delta}$ is $R$ embedded diagonally in $R^n$.
\end{rem}


\begin{example}\label{example_split_square_root}
As a first example, we treat degree split degree $2$ extensions of perfectoid rings when $p\neq 2$, in the sense that we consider
\begin{equation*}
    S \defeq R[t]/(t^2-d)
\end{equation*}
where $d\in R$ is a non-zero divisor which admits a square root $\delta$ in $R$. Let's compute the discriminant in this case. The product of the difference of the roots $\delta,-\delta$ is equal to $(2\delta)\cdot(-2\delta)=-4d$. As $p\neq 2$ we have that $4$ is invertible.

We now want to show that
\[R/dR\to S/dS\]
is an isomorphism at perfectoidization. But note that it is an isomorphism at reduction -- therefore this map is an isomorphism of ($p$-complete) arc sheaves, which implies the result by \cite[Corollary 8.11]{bhatt2022Prismsa}. Perhaps more simply, one could also argue using the fact that perfecoidizations of both rings are discrete by \cite[Theorem 10.9]{bhatt2022Prismsa} and that perfectoid rings are reduced, and then using the universal property of perfectoidization.
Then using \Cref{split_ramification_is_point} we see that the perfectoidization of $S$ is identified with the map of \(R\)-algebras
    \begin{equation*}
        S \xrightarrow{\ev}R_{\Delta}+(d)_{\pfd} \times (d)_{\pfd}
    \end{equation*}
    to the subring of \(R \times R\) where \((d)_{\pfd} \defeq \ker(R \to (R/dR)_{\pfd})\).
    Note that because $p\neq 2$, $(2d)_{\pfd}=(d)_{\pfd}$.
    
\end{example}

\begin{example}\label{example_split_mth_roots} In this example, we continue the explore the case treated in \Cref{mth_roots_section}. Note that this is a direct generalization of \Cref{example_split_square_root}. Namely, let $R$ be a perfectoid ring, $m\in \N$ such that $(m,p)=1$. We take $r\in R$ a non-zero divisor such that $t^m-r$ splits as a polynomial in $R[t]$ and consider
\[S=R[t]/(t^m-r).\]

We first compute the discriminant of $t^m-r$. Using the properties recalled in \Cref{disc_def}, and that the derivative of this polynomial is $mt^{m-1}$ we get that the discriminant $\Res(t^m-r,mt^{m-1})$ is $m^mr^{m-1}$ up to a sign. As $(m,p)=1$, $m$ is invertible in $R$ and we conclude that $R\to R[t]/(t^m-r)$ is étale when inverting $r$.

Now, similarly to \Cref{example_split_square_root}, as $R/(r)\to R[t]/(r,t^m-r)$ is an isomorphism at reduction, we conclude that it is also an isomorphism at perfectoidization. This implies that we can use \Cref{split_ramification_is_point} to compute the perfectoidization of this algebra. Explicitly 
\[S_{\pfd}=\Facst{r}(\underline{m},R).\]
\end{example}

\begin{example}\label{example_pnthroots_unit}
    Fix a $n\in \N$ and suppose that $R$ is $p$-torsion free. In this example we treat the case of 
    $$R\to S \defeq R[t]/(t^{p^n}-u)$$
    where $u\in R^{\times}$ is a unit such that $t^{p^n}-u$ splits as a polynomial in $R[t]$. Fix $u^{1/p^n}$ a root. As in \Cref{example_split_mth_roots}, we can compute the discriminant which is $(p^n)^{p^n}u^{p^n-1}$ up to a sign. Therefore the above algebra is étale away from $p$. 
    
    Note also that $R/pR\to R[t]/(p,t^{p^n}-u)=R/pR[t]/(t-u^{1/p^n})^{p^n}$. Therefore this map is an isomorphism at reduction.

    As a consequence, as in \Cref{example_split_square_root}, we can apply \Cref{split_ramification_is_point} to understand the perfectoidization. Namely
    \[S_{\pfd}=\Facst{p}(\underline{p^n},R).\]
\end{example}

\begin{example}\label{example_general_pnth_roots}
    We now generalize \Cref{example_pnthroots_unit}. Let $R$ be a $p$-torsion free perfectoid ring, $r\in R$ a non-zero divisor.
    Assume that $(t^{p^n}-r)$ is a split polynomial. 
    We want to study the perfectoidization of
    \[S=R[t]/(t^{ p^n}-r).\]

    Using the same calculation as in \Cref{example_split_mth_roots}, the discriminant of $t^{p^n}-r$ is $(p^n)^{p^n}r^{p^n-1}$ up to a sign.
    Therefore, $R\to R[t]/(t^{p^n}-r)$ is étale after inverting $pr$.

    We now claim that 
    \[R/(pr)\to R[t]/(pr,t^{p^n}-r)\]
    is an isomorphism at perfectoidization. For this we propose a proof that shows directly that this map is an isomorphism in the $p$-complete arc topology. 
    
    Let $R/(pr)\to V$ be a map to perfectoid valuation ring of rank 1. We want to show that such a map uniquely lifts to $R/(pr)\to R[t]/(pr,t^{ p^n}-r)$. 
Suppose first that $V$ is $p$-torsion free. Then $r$ is sent to zero in $V$. Therefore it suffices to show that $R/(r)\to V$ lifts uniquely to $R/(r)\to R[t]/(r,t^{ p^n}-r)$. We note that sending $t$ to zero now gives an extension. Suppose there is one extension -- note that $t^{ p^n}$ is sent to zero by construction through this extension. But as $V$ is reduced, $t$ is also sent to zero. Unicity now follows.
If $V$ is $p$-torsion, then $p$ is sent to zero and the map factors through $\overline{R}\to V$ where $\overline{R}$ is the perfection (reduction) of $R/pR$. So it suffices to show that $\overline{R}\to V$ admits a unique lifts to $\overline{R}[t]/(t^{p^n}-r)$. Using that $V$ is perfect, an extension is given by sending $t$ to a $p^n$-th root of the image of $r$ -- this extension is necessarily unique by the perfect aspect of $V$.

It now follows that we can apply \Cref{split_ramification_is_point} to understand the perfectoidization of $R[t]/(t^{p^n}-r)$. Namely
\[S_{\pfd}=\Facst{pr}(\underline{p^n},R).\]
\end{example}

\begin{example}\label{mupn_pfd}
    
In \Cref{example_pnthroots_unit}, we can take $u=1$. In this case, this is the Hopf-algebra of the finite flat $R$-group scheme $\mu_{p^n,R}$. The evaluation map at $p^n$-th roots of unity can be in this case reinterpreted as a faithfully flat morphism of group schemes
\[\underline{\Z/p^n\Z}_R\to \mu_{p^n,R}.\]
One intriguing feature of perfectoidization is that it preserves the group structure. Namely if $R$ is a perfectoid ring and $R\to S$ is a $R$-Hopf-algebra that admits a discrete perfectoidization as a ring, then $S_{\pfd}$ is canonically an $R$-Hopf-algebra finite and finitely presented group schemes over a perfectoid ring an affine group scheme because $(-)_{\pfd}$ is symmetric monoidal over a perfectoid ring. We can even understand explicitly the Hopf-algebra structure on the perfectoidization. Namely, as we have a commutative diagram
\[\begin{tikzcd}
\underline{\Z/p^n\Z}_R \arrow[r]                                 & {\mu_{p^n,R}}                                     \\
\underline{\Z/p^n\Z\times \Z/p^n\Z}_R \arrow[u, "\mult"] \arrow[r] & {\mu_{p^n,R}\times_R\mu_{p^n,R}} \arrow[u, "\mult"]
\end{tikzcd}\]
we can apply the perfectoidization functor to deduce that the Hopf-algebra structure on 
\[\Facst{p}(\Z/p^n\Z,R)\]
is induced by the group law on $\Z/p^n\Z$.
\end{example}





    

\subsection{Application to perfectoidization of semiperfectoids} We use the results of \Cref{split_finite_examples} to study the perfectoidization of a semiperfectoid, also in the split case.

\begin{nota}\label{acst-continous-notation}As in \Cref{acst-notation}, we introduce some notation. Let $X$ be a profinite set. Say $R$ is a perfectoid, implicitly equipped with the $p$-adic topology. If $d\in R$, we denote by
\[\Cacst{d}(X,R)\]
the ring (pointwise operations) of continuous functions from $X$ with its profinite topology and $R$ with the $p$-adic topology.

Also, consider the following situation: $R_0$ is a ring, that we consider without topology (with its discrete topology). Then
\[\Cont(X,R_0)=\varinjlim_{X\to X_i}\Cont(X_i,R_0)\]
where $X\to X_i$ ranges over all continuous finite quotients or to any set of finite quotients such that $X=\varprojlim_{i}X_i$. Also, the $p$-completion of $\Cont(X,R_0)$ is $\Cont(X,R)$ if $R$ denotes the $p$-completion of $R_0$ equipped with the $p$-adic topology. Indeed $\Cont(X,R_0)$ modulo $p^n$ is $\Cont(X,R_0/p^n R_0)$ which can be seen using the description as an union above, using that $X_i$'s are finite. Now giving a continuous function from $X$ to $R$ is the same as giving a collection of continuous functions from $X$ to $R_0/p^nR_0$, thus giving the claim.
    
\end{nota}
\begin{nota}\label{semipfd_split_calculation_setup}We lay some setup for \Cref{semipfd_split_calculation}. Let $R$ be a  $p$-torsion free perfectoid ring. Let $r\in R$ be a non-zero divisor such that for every $n\in \N$ the polynomial $t^{p^n}-r$ splits with roots $\alpha_{n,1},\dots, \alpha_{n,p^n}\in R$. Assume furthermore that for every $n\in \N$, $k\in \underline{p^n}$ and $i\in \{0,\dots,p-1\}$ we have 
\[\alpha_{{n+1},k+ip^n}^p=\alpha_{n,k}.\]
Now, using the surjective map $\underline{p^{n+1}}\to \underline{p^n}$ coming from the surjective maps at residue classes modulo $p^{n+1}$ and $p^n$-respectively, we get a diagram of $pr$-isogenies
\[\begin{tikzcd}
{R[t^{1/p^n}]/(t-r)} \arrow[r, "t\mapsto t^p"] \arrow[d] & {R[t^{1/p^{n+1}}]/(t-r)} \arrow[d] \\
{\Fun(\underline{p^n},R)} \arrow[r]                    & {\Fun(\underline{p^{n+1}},R)}               
\end{tikzcd}\]
where the vertical maps are given by evaluation of $t$ at $\alpha_{n,k}$'s and $\alpha_{n+1,k'}$'s respectively. We realize 
\[\Z_p=\varprojlim_{n\in \N}\underline{p^n}\]
and regard it merely as a profinite set. Note that this set is in bijection with a compatible choices of $p$-power roots of $r$ in $R$ by sending a $(x_n)\in \varprojlim_{n\in \N}\underline{p^n}$ to $(\alpha_{n,x_n})\in R^{\flat}$.
    
\end{nota}

\begin{proposition}\label{semipfd_split_calculation}Let $R$ be a perfectoid ring. Let $r\in R$ be a non-zero divisor as in \Cref{semipfd_split_calculation_setup}. Then the map
\[R\abracket{t^{1/p^{\infty}}}/(t-r)\to \Cacst{pr}(\Z_p,R)\]
sending $t^{1/p^n}$ to the function 
$f_n\colon \Z_p\to \underline{p^n}\xrightarrow{k\mapsto \alpha_{n,k}} R$
is a perfectoidization map.
This is the $p$-completed colimit of the perfectoidizations  
$$R[t^{1/p^n}]/(t-r)\to \Facst{pr}(\underline{p^n},R)$$
from \Cref{split_ramification_is_point}.
    
\end{proposition}
\begin{proof} Let $A_n=R[t^{1/p^n}]/(t-r)$.
    We get for every $n\in \N$ a pullback square
    \[\begin{tikzcd}
{(A_n)_{\pfd}=\Facst{pr}(\underline{p^n},R)} \arrow[r] \arrow[d] & {\Fun(\underline{p^n},R)} \arrow[d]   \\
R/(pr)_{\pfd} \arrow[r]                                            & {\Fun(\underline{p^n},R/(pr)_{\pfd})}
\end{tikzcd}\]
realizing the perfectoidization as explained in \Cref{split_ramification_is_point}. 

Taking the (uncompleted) directed colimit on $n$, we get that the map 
$$R[t^{1/p^{\infty}}]/(t-r)\to \Cont(\Z_p,R_0)$$
where $R_0=R$ but equipped with the discrete topology, is an isomorphism away from $pr$. Now, using \Cref{acst-continous-notation} and \cite[Proposition 2.7]{struc_pfd} we have a pullback square
\[\begin{tikzcd}
\left(R\abracket{t^{1/p^{\infty}}}/(t-r)\right)_{\pfd} \arrow[r] \arrow[d] & {\Cont(\Z_p,R)} \arrow[d]  \\
R/(pr)_{\pfd} \arrow[r]                                                                    & {\Cont(\Z_p,R/(pr)_{\pfd})}
\end{tikzcd}\]
which now concludes.
\end{proof}

\begin{rem}We explain the claim in \cite[Remark 4.18]{navarrochafloque2026perfectoid}. Namely let $R$ be a $p$-torsion free perfectoid ring. We consider the perfectoidization of the semiperfectoid 
\[R\abracket{t^{1/p^\infty}}/(t-1).\]
in the setup from \Cref{semipfd_split_calculation}.
Assume that there exists a system of compatible system of $p$-power roots of 1 in $R$ and make choices as in \Cref{semipfd_split_calculation_setup} for roots of $t^{p^n}-1$.

Note that $\sigma\defeq \lim_{n\to \infty}(t^{1/p^n}-1)^{p^n}$ is an element of $R\abracket{t^{1/p^{\infty}}}$ which admits a compatible sequence of $p$-power roots given by shifting the limit that we denote by $\sigma^{\flat}$. 
We also have the property that
\[t-1\equiv \sigma \mod pR\abracket{t^{\infty}}.\]
Moreover, when viewing $R^{\flat}\abracket{t^{1/p^{\infty}}}=\varprojlim_{(-)^p}R/pR[t^{1/p^{\infty}}]$, we see that we can identify $\sigma^{\flat}$ to $t-1\in R^{\flat}\abracket{t^{1/p^{\infty}}}$.


We claim that $\sigma\not \in (t-1)_{\pfd}$. Namely if the opposite was true, then using \cite[Lemma 4.6]{struc_pfd} we would have $\sigma^{\flat}\in (t-1)_{\pfd}^{\flat}$ and therefore $t-1\in ((t-1)_{\pfd})^{\flat}\subseteq R^{\flat}\abracket{t^{1/p^{\infty}}}$. Using \Cref{semipfd_split_calculation} we see that $(t-1)_{\pfd}^{\flat}$ is the kernel of
\begin{equation}\label{tilt_t-1_pfd}
    R^{\flat}\abracket{t^{1/p^{\infty}}}\to \Cacst{p^{\flat}}(\Z_p,R^{\flat}).
\end{equation}
Let us explain precisely the map in \Cref{tilt_t-1_pfd}: $p^{\flat}$ denotes any $p$-compatible sequence of $p$-power roots of some element $pu\in R$ where $u$ is a unit. The completion on the left is a $p^{\flat}$-completion. The ring $R^{\flat}$ is equipped with the $p^{\flat}$-topology which coincides with the limit topology when we see $R^{\flat}=\varprojlim_{(-)\mapsto(-)^p} R$ and equip $R$ with the $p$-adic topology. Using this identification and the universal property of limits we see that $\Cont(\Z_p,R^{\flat})=\Cont(\Z_p,R)^{\flat}$. Here ``$p^{\flat}-\acst$'' means that we are looking at functions that are constant modulo $\sqrt{p^{\flat} R^{\flat}}$. Recall that tilting a surjection is still surjective (see the proof of \cite[Lemma 4.6]{struc_pfd}) -- the map in \Cref{tilt_t-1_pfd} is the tilt of the morphism of perfectoid $R$-algebras $R\abracket{t^{1/p^{\infty}}}\to R\abracket{t^{1/p^{\infty}}}/(t-1)_{\pfd}$.


This surjective map is given by sending $t$ to the function $\Cont(\Z_p,R^{\flat})=\Cont(\Z_p,R)^{\flat}$ given by the sequence of functions from \Cref{semipfd_split_calculation} $(f_n)_{n\in \N}\in \Cont(\Z_p,R)^{\flat}$ satisfying $f_{n+1}^p=f_n$ for every $n\in \N$. In particular, we explain below why $t$ is not sent to the constant function $1$. Say with our identification $\Z_p=\varprojlim_{n\in \N}\underline{p^n}$, that $1\in \Z_p$ correspond to the system given by $1\in \underline{p^n}$ for each $n$ and similarly for $2$. Then in $R^{\flat}=\varprojlim_{(-)\mapsto(-)^p} R$
$$t(1)=(f_n(1))\neq (f_n(2))=t(2)$$
because $f_1(1)=\alpha_{1,1}$ and $f_2(2)=\alpha_{1,2}$ which are distinct roots of $t^p-1$ in the notations of \Cref{semipfd_split_calculation_setup}. This concludes.

\end{rem}

\begin{example}\label{zp1_pfd} As a follow up to \Cref{mupn_pfd} we mention that using the same argument the Hopf-algebra structure on 
\[R\abracket{t^{1/p^{\infty}}}/(t-1)\]
which is the $p$-complete ring of functions of the $p$-complete formal $R$-group scheme 
$$\Z_p(1)_R=\varprojlim_{n\in \N}\mu_{p^n}$$
also induces an Hopf-algebra structure on its perfectoidization. Namely, in the setup of \Cref{semipfd_split_calculation} it's perfectoidization is identified to
\[\Cacst{p}(\Z_p,R)\]
and the Hopf-algebra structure is induced by the group law on $\Z_p$.
\end{example}

\end{document}